\documentclass[10pt]{amsart}
\usepackage{amsmath}
\usepackage{amsxtra}
\usepackage{amscd}
\usepackage{amsthm}
\usepackage{amsfonts}
\usepackage{amssymb}
\usepackage{eucal}
\usepackage{mathabx}
\usepackage[matrix,arrow,curve]{xy}
\usepackage{pb-diagram}
\usepackage{comment}
\usepackage{slashbox}
\usepackage{array}
\newcommand{\fg}{{\mathfrak{g}}}
\newcommand{\fh}{{\mathfrak{h}}}
\newcommand{\fz}{{\mathfrak{z}}}

\newcommand{\fgl}{{\mathfrak{gl}}}
\newcommand{\tr}{{\rm tr} \,}
\newcommand{\diff}{{\rm d} \,}
\newcommand{\End}{{\rm End} \,}

\newcommand{\diag}{{\rm diag} \,}
\newcommand{\rk}{{\rm rk} \,}
\newcommand{\gr}{{\rm gr} \,}
\newcommand{\sgn}{{\rm sgn} \,}
\newcommand{\qdet}{{\rm qdet} \,}

\newcommand{\BC}{{\mathbb{C}}}
\newcommand{\BR}{{\mathbb{R}}}
\newcommand{\BP}{{\mathbb{P}}}
\newcommand{\BZ}{{\mathbb{Z}}}

\newcommand{\ol}{\overline}

\newtheorem{defn}[subsection]{Definition}
\newtheorem{thm}[subsection]{Theorem}
\newtheorem{lem}[subsection]{Lemma}
\newtheorem{prop}[subsection]{Proposition}

\newtheorem{cor}[subsection]{Corollary}

\newtheorem{fact}[subsection]{Fact}
\newtheorem*{conjj}{Conjecture}
\newtheorem*{rem}{Remark}

\title{Degeneration of Bethe subalgebras in the Yangian of $\mathfrak{gl}_n$}
\author{Aleksei Ilin and Leonid Rybnikov}

\begin{document}
\maketitle

\begin{abstract}
We study degenerations of Bethe subalgebras $B(C)$ in the Yangian $Y(\mathfrak{gl}_n)$, where $C$ is a regular diagonal matrix. We show that closure of the parameter space of the family of Bethe subalgebras, which parametrizes all possible degenerations, is the Deligne-Mumford moduli space of stable rational curves $\overline{M_{0,n+2}}$. All subalgebras corresponding to the points of $\overline{M_{0,n+2}}$ are free and maximal commutative. We  describe explicitly the ``simplest'' degenerations and show that every degeneration is the composition of the simplest ones. The Deligne-Mumford space $\overline{M_{0,n+2}}$ generalizes to other root systems as some De Concini-Procesi resolution of some toric variety. We state a conjecture generalizing our results to Bethe subalgebras in the Yangian of arbitrary simple Lie algebra in terms of this De Concini-Procesi resolution. 
\end{abstract}
\makeatletter
\@setabstract
\makeatother

\section{Introduction}
\subsection{Bethe subalgebras in the Yangian.} Yangian for $\fgl_n$ is the associative algebra, historically one of the first examples of {\em quantum groups}. The Yangian $Y(\fgl_n)$ is a Hopf algebra
deforming the enveloping algebra $U(\fgl_n[t])$, where
$\fgl_n[t]$ is the (infinite dimensional) Lie algebra of $\fgl_n$-valued polynomials. 
This algebra was considered in the works of L. Fadeev and St.-Petersburg school in the relation with the inverse scattering method, see e.g. \cite{TF,T}.
There is a family $B(C)$ of commutative subalgebras in $Y(\fgl_n)$ parameterized by complex matrices $C\in {\rm Mat}_n$ called Bethe subalgebras. This family originates from the integrable models in statistical mechanics and algebraic Bethe ansatz. For details and links on Yangians, we refer the reader to the survey \cite{molev} and to the book \cite{molev2} by A.~Molev. 

Denote by $T$ the maximal torus in $GL_n$ i.e. the subgroup of diagonal matrices in $GL_n$. In the present paper we restrict ourselves to Bethe subalgebras with $C\in T$. Let $T^{reg}$ be the set of regular elements of the torus, i.e. the set of matrices from $T$ with pairwise distinct eigenvalues. We will frequently use the embedding $GL_n\subset \fgl_n={\rm Mat}_n$ and regard $C$ as an element of the Cartan subalgebra $\fh\subset\fgl_n$.
In \cite{nazol} Nazarov and Olshanski showed that $B(C)$ is a free polynomial algebra and that it is a maximal commutative subalgebra in $Y(\fgl_n)$ for all $C \in T^{reg}$.
For non-regular $C\in T\backslash T^{reg}$, the subalgebra $B(C)$ becomes smaller. But there is a natural way to assign a commutative subalgebra of the same size as for $C \in T^{reg}$ to any $C_0\in \fh\backslash T^{reg}$ by taking some {\em limit} of $B(C)$ as $C\to C_0$. For example, one can get the Gelfand-Tsetlin subalgebra of the Yangian as the $t\to0$ limit of some $1$-parametric family of Bethe subalgebras with $C(t)\in T^{reg}$ for $t\ne0$ and $C(0)=E_{11}$. In general, such limit subalgebra $\lim\limits_{C\to C_0}B(C)$ is not unique since it depends on the $1$-parametric family $C(t)$ such that $C(0)=C_0$. Our goal is to describe all possible limit subalgebras.

The images of Bethe subalgebras in the universal enveloping algebra $U(\fgl_n)$ under the evaluation homomorphism are known as ``shift of argument subalgebras''. The problem of describing all limits for the  shift of argument subalgebras was posed by Vinberg in late 1990-s. The answer was given by V.Shuvalov in \cite{shuvalov} and later in more algebro-geometric terms by L.~Aguirre, G.~Felder and A.~Veselov in \cite{AFV}. Their description is roughly as follows. Shift of argument subalgebras themselves are parametrized by regular diagonal $n\times n$-matrices up to proportionality and up to adding a scalar matrix. The latter can be regarded as the space of configurations of $n$ pairwise distinct points on the complex line. It turns out that the limit shift of argument subalgebras are parametrized by the Deligne-Mumford closure of this space \cite{AFV,shuvalov}, all limit subalgebras are free \cite{shuvalov} and maximal commutative \cite{taras}, and moreover there is an explicit inductive procedure generating the limit subalgebras from smaller shift of argument subalgebras assigned to smaller $n$. It is natural to expect a similar description for limit Bethe subalgebras in the Yangian.

\subsection{Limits of Bethe subalgebras.} The limit subalgebras can be defined in purely algebro-geometric terms (we will do this in Section~2). Roughly, the construction of Bethe algebras can be regarded as a regular map from $T^{reg}$ to the ``Grassmannian'' of subspaces in $Y(\fgl_n)$ of the same ``dimension'' as $B(C)$. The space $T^{reg}$ is noncompact while the Grassmannian is (in appropriate sense) compact. So we can take the closure of the image of this map and obtain new subalgebras with the same Poincare series. We call such subalgebras {\em limit subalgebras.}

Since the subalgebras $B(C)$ do not change under dilations of $C$, the parameter space for the family $B(C)$ is the quotient $T^{reg}/\mathbb{C}^*$ of the set of regular elements of the torus by the subgroup of scalar matrices. Following \cite{AFV} we can regard the space $T^{reg}/\mathbb{C}^*$ as the moduli space $M_{0,n+2}$ of rational curves with $n+2$ marked points (we can assign to a matrix $C$ with the eigenvalues $z_1,\ldots,z_n$ the curve $\mathbb{P}^1$ with the marked points $0,z_1,\ldots,z_n,\infty$). Therefore the limit subalgebras of the family $B(C)$ are parametrized by some compactification of the space $M_{0,n+2}$. The main results of the present paper can be summarized as follows:

\medskip

\noindent {\bf Main Theorem.} \emph{The closure of $T^{reg}/\mathbb{C}^*$ which parametrizes the limit subalgebras is isomorphic to the Deligne-Mumford compactification $\overline{M_{0,n+2}}$. All limit subalgebras are free polynomial algebras and maximal commutative subalgebras in $Y(\fgl_n)$.}

\medskip

In fact we will describe all limit subalgebras as products of smaller Bethe subalgebras and some shift of argument subalgebras (which are the images of Bethe subalgebras in the universal enveloping algebra $U(\fgl_n)$ under the evaluation map).

It is natural to expect this result from the perspective of bispectral duality of Mukhin, Tarasov and Varchenko \cite{MTV} which relates the image of Bethe algebra in a tensor product of evaluation representations with the algebra of higher Hamiltonians of the trigonometric Gaudin model for $\fg=\fgl_k$. \footnote{We thank Evgeny Mukhin for pointing this out.} The latter is the image of a Gaudin subalgebra in the tensor product of $n+1$ copies of the universal enveloping algebra of $\fgl_k$. On the  other hand in \cite{ryb2} we prove that the closure of the parameter space for this family of Gaudin algebras is $\overline{M_{0,n+2}}$. The problem in this approach is that we have to deal with the images of Bethe subalgebras in some specific representations of the Yangian, so the closure of the parameter space could be different from that of the subalgebras in the Yangian itself. We use different approach based on shift of argument subalgebras and the Olshanski centralizer construction of the Yangian. 

The main idea of our proof is to reduce the problem of describing the closure to the similar one for the family of shift of argument subalgebras in $U(\fgl_N)$ where $N$ is big enough. For this we use on the one hand the Olshanski centralizer construction \cite{olsh} which approximates the Yangian $Y(\fgl_n)$ by the centralizer subalgebras of the form $U(\fgl_{n+k})^{\fgl_k}$, and, on the other hand, the results of Shuvalov \cite{shuvalov} and Tarasov \cite{taras} which describe the limit subalgebras for the family of shift of argument subalgebras.

\subsection{Shift of argument subalgebras.}
Let $\hat F(C)\subset U(\fgl_n)$ be the image of $B(C)\subset Y(\fgl_n)$ under the evaluation homomorphism $Y(\fgl_n)\to U(\fgl_n)$. The subalgebra $\hat F(C)$ does not change after adding a scalar matrix to $C$. The associated graded of $\hat F(C)$ is a Poisson commutative subalgebra $F(C)\subset S(\fgl_n)$ known as \emph{shift of argument subalgebra} since it is generated by the adjoint invariants from $S(\fgl_n)$ shifted by $tC$ for all $t\in\BC$, see \cite{nazol}.

Suppose $C$ is a diagonal matrix with pairwise distinct eigenvalues $z_1,\ldots,z_n$. Then the algebra $F(C)$ contains the quadratic elements $H_i:=\sum\limits_{j\ne i}\frac{e_{ij}e_{ji}}{z_i-z_j}$ which are the coefficients of (an appropriate version of) the $KZ$ connection. Moreover both $F(C)$ and $\hat F(C)$ are uniquely determined by the subspace $Q_C\subset S(\fgl_n)$ which is the linear span of the $H_i$'s. Note that $H_i$ do not change under simultaneous affine transformations of the $z$'s, hence the space of parameters of the corresponding shift of argument subalgebras is naturally the configuration space of $n$ pairwise distinct points on the affine line or, equivalently, the configuration space $M_{0,n+1}$ of $n+1$ pairwise distinct points on the projective line. From the results of Aguirre, Felder and Veselov \cite{AFV} it follows that the closure of the family of subspaces $Q_C\subset S(\fgl_n)$ is the Deligne-Mumford compactification $\overline{M_{0,n+1}}$, which is crucial for our proof.

\subsection{Centralizer construction.} Let $A_0=\BC[\mathcal{E}_1,\mathcal{E}_2, \ldots, ]$ be the filtered polynomial algebra of infinitely many generators such that $\deg \mathcal{E}_i=i$. The Olshanski centralizer construction \cite{olsh} is the collection of surjective homomorphisms of filtered algebras $Y(\fgl_n)\otimes A_0\to U(\fgl_{n+k})^{\fgl_k}$ generalizing the evaluation map. The intersection of kernels of such homomorphisms is known to be zero, so this collection of homomorphisms can be regarded as an {\em an asymptotic isomorphism} i.e. for each filtered component of $Y(\fgl_n)\otimes A_0$ there is $K\in\BZ$ such that for any $k>K$ the restriction of the above homomorphism to this filtered component is an isomorphism. The idea of our proof is to analyze the parameter spaces of the images of $B(C)$ under the centralizer construction maps. Since centralizer construction is an asymptotic isomorphism the closure of the parameter space stabilizes for $k>>0$. 

\subsection{Plan of the proof.} We prove that the image of $B(C)\otimes A_0$ in the centralizer algebra $U(\fgl_{n+k})^{\fgl_k}$ is contained in some non-regular shift of argument subalgebra $\hat F(C^{(k)})$ (Proposition~\ref{incl}) and moreover $\hat F(C^{(k)})$ is asymptotically isomorphic to $B(C)\otimes A_0$. We show that the closure of the parameter space for the subalgebras $\hat F(C^{(k)})$ is $\overline{M_{0,n+2}}$ (realized as a subvariety in $\overline{M_{0,n+k+1}}$ which parametrizes all limit shift of argument subalgebras in $U(\fgl_{n+k})$). Since the closure of the parameter space for $\hat F(C^{(k)})$ does not depend on $k$, the closure of the parameter space for Bethe subalgebras in $Y(\fgl_n)$ is  $\overline{M_{0,n+2}}$ as well.

Next, we deduce from the results of Shuvalov and Tarasov that any limit of the algebras $\hat F(C^{(k)})$ is a free polynomial algebra and a maximal commutative subalgebra in $U(\fgl_{n+k})^{\fgl_k}$. Since $B(C)\otimes A_0$ is asymptotically isomorphic to $\hat F(C^{(k)})$, the same is true for limit Bethe subalgebras corresponding to points of $\overline{M_{0,n+2}}$ (Theorem \ref{result}). Using Shuvalov's description of limit shift of argument subalgebras we explicitly describe the simplest limits corresponding to generic points of codimension $1$ strata in   $\overline{M_{0,n+2}}$ in terms of Bethe subalgebras for smaller Yangians and shift of argument subalgebras for smaller Lie algebras (Theorem \ref{result2}). Iterating this procedure we obtain the explicit description of all limit Bethe subalgebras.

\subsection{Generalization to Yangians of other types.}
To any semisimple Lie algebra $\fg$ (and even more generally, to any Kac-Moody algebra $\fg$) one can assign the Yangian $Y(\fg)$ , the quantum group generated by the rational R-matrix of $\fg$. To any element $C$ of the Cartan torus $T$ of the corresponding Lie group $G$ one can assign a Bethe subalgebra generated by traces of the products of $C$ with the R-matrix in all integrable representations. This is a commutative subalgebra which is expected to be maximal for regular $C$. The variety parameterizing all possible limits of such commutative subalgebras is a compactification of the set $T^{reg}$ of regular elements of the torus (well-defined as a pro-algebraic scheme). Our Theorem~\ref{result} states that in the case $\fg=\fgl_n$ it is $\overline{M_{0,n+2}}$. The natural generalization of this statement to Lie algebras of other types is the De Concini -- Procesi closure \cite{DCP} of the complement to the following arrangement of subvarieties in a toric variety. Consider the toric variety $X$ (acted on by the maximal torus $T\subset G$) which corresponds to the fan determined by the root hyperplanes. Equivalently, $X$ is the closure of a generic $T$-orbit in the flag variety $G/B$. We can regard $T^{reg}$ as a complement of an arrangement of hypersurfaces in $X$. Following De Concini and Procesi \cite{DCP}, one can construct a compactification $M_{\fg}$ of $T^{reg}$ by blowing up all indecomposable intersections of the hypersurfaces in $X$.

\begin{conjj} $M_{\fg}$ is the parameter space for limit subalgebras of the family of Bethe subalgebras in the Yangian $Y(\fg)$.
\end{conjj}

\begin{rem}
Note that $\overline{M_{0,n+2}}$ is the De Concini--Procesi closure of $T^{reg}$ for $\fg=\fgl_n$.
\end{rem}

\begin{rem}
For infinite root systems (say, for affine Kac-Moody) $M_\fg$ is not well-defined as an algebraic scheme but is still well-defined as a pro-algebraic scheme. On the other hand, in this case the Bethe subalgebras themselves live not in the Yangian but rather in its completion which is an inverse limit of some quotients of the Yangian. So we can generalize our conjecture to Yangians of infinite-dimensional Lie algebras by stating that two inverse limits are isomorphic.
\end{rem}

\subsection{Application to crystals.} By analogy with the shift of argument algebras and Gaudin algebras, we expect that for real values of the parameter the corresponding (limit) Bethe algebra has simple spectrum in any irreducible finite-dimensional (or integrable) representation of the Yangian. Following \cite{ryb2}, one can assign a covering of $M_\fg(\BR)$ to any irreducible representation of $Y(\fg)$. The fiber of such covering is just the set of joint eigenlines for the elements of a Bethe algebra in this representation. For Kirillov-Reshetikhin modules, we expect a natural bijection of the fiber of this covering with the corresponding Kirillov-Reshetikhin crystal hence obtaining an action of the fundamental group $\pi_1(M_\fg(\BR))$ on the crystal. We also expect that this action can be described in purely combinatorial terms.

\subsection{The paper is organized as follows.}  In section~2 we recall some known facts about Yangians  and Bethe subalgebras. In section~3 we recall the definition of the moduli space $\overline{M_{0,n+1}}$. In section 4 we discuss some facts about shift of argument subalgebras. In Section 5 we formulate the main theorems. In section 6 we prove the main theorems.

\subsection{Acknowledgements}
We thank Michael Finkelberg and Alexander Molev for helpful remarks and references. We thank the referee for careful reading of the first version of the text. This work has been funded by the  Russian Academic Excellence Project '5-100'. The work was finished during A.I.'s internship at MIT supported by NRU HSE. A.I. is deeply indebted to MIT and especially to R.~Bezrukavnikov and P.~Etingof, for providing warm hospitality and excellent working conditions. The work of A.I. and L.R. was supported in part by the Simons Foundation. The work of L.R. has been supported by the Russian Science Foundation under grant 16-11-10160. 

\section{Yangians and Bethe subalgebras.}
\subsection{Definitions.} We follow the notations and conventions of \cite{molev2}.
\begin{defn}
{\itshape Yangian for $\fgl_n$} is a complex unital associative algebra with countably many generators $t_{ij}^{(1)}, t_{ij}^{(2)}, \ldots $ where $1 \leqslant i,j \leqslant n$, and the defining relations
$$ [t_{ij}^{(r+1)},t_{kl}^{(s)}] - [t_{ij}^{(r)},t_{kl}^{(s+1)}] = t_{kj}^{(r)}t_{il}^{(s)} - t_{kj}^{(s)}t_{il}^{(r)}, $$
where $r,s \geqslant 0$ and $t_{ij}^{(0)} = \delta_{ij}$. This algebra is denoted by $Y(\fgl_n)$.
\end{defn}
It is convenient to consider the formal series
$$t_{ij}(u) = \delta_{ij} + t_{ij}^{(1)} u^{-1} + t_{ij}^{(2)} u^{-2} + \ldots \in Y(\fgl_n)[[u^{-1}]].$$

We denote by $T(u)$ the matrix whose $ij$-entry is $t_{ij}(u)$.
We regard this matrix as the following element of $Y(\fgl_n)[[u^{-1}]] \otimes \rm{End} \ \mathbb{C}^n$:
$$T(u) = \sum_{i,j=1}^n t_{ij}(u) \otimes e_{ij},$$
where $e_{ij}$ stands for the standard matrix units.

Consider the algebra
$$Y(\fgl_n)[[u^{-1}]] \otimes (\rm{End} \ \mathbb{C}^n)^{\otimes n}.$$
For any $a \in \{1, \ldots ,n\}$ there is an embedding $$
i_a: Y(\fgl_n)[[u^{-1}]] \otimes \rm{End} \ \mathbb{C}^n\to Y(\fgl_n)[[u^{-1}]] \otimes (\rm{End} \ \mathbb{C}^n)^{\otimes n}
$$
which is an identity on $Y(\fgl_n)[[u^{-1}]]$ and embeds $\rm{End} \ \mathbb{C}^n$ as the $a$-th tensor factor in $(\rm{End} \ \mathbb{C}^n)^{\otimes n}$. Denote by $T_a(u)$ the image of $T(u)$ under this embedding.

The symmetric group $S_n$ acts on $Y(\fgl_n)[[u^{-1}]] \otimes (\rm{End} \ \mathbb{C}^n)^{\otimes n}$ by permuting the tensor factors. This action factors through the embedding $S_n\hookrightarrow (\rm{End} \ \mathbb{C}^n)^{\otimes n}$ hence the group algebra $\BC[S_n]$ is a subalgebra of $Y(\fgl_n)[[u^{-1}]] \otimes (\rm{End} \ \mathbb{C}^n)^{\otimes n}$. Let $S_m$ be the subgroup of $S_n$ permuting the first $m$ tensor factors. Denote by $A_m$ the antisymmetrizer $$\sum_{\sigma \in S_m} (-1)^{\sigma}\sigma \in \mathbb{C}[S_m]\subset Y(\fgl_n)[[u^{-1}]] \otimes (\rm{End} \ \mathbb{C}^n)^{\otimes n}.$$
Let $C$ be a $\mathbb{C}$-valued $n \times n$ matrix. For any $a \in \{1, \ldots ,n\}$ denote by $C_a$ the element $i_a(1\otimes C)\in Y(\fgl_n)[[u^{-1}]] \otimes (\rm{End} \ \mathbb{C}^n)^{\otimes n}$.
\begin{defn}
For any $1 \leqslant k \leqslant n$ introduce the series with coefficients in $Y(\fgl_n)$ by
$$\tau_k(u,C) = \frac{1}{k!} \tr A_k  C_1 \ldots C_k T_1(u) \ldots T_k(u-k+1),$$
where we take the trace over all $k$ copies of $ \rm{End} \ \mathbb{C}^n$.
\end{defn}

\begin{rem}
$\tau_n(u,C)=\det C\ \tau_n(u,E)$, and hence, does not depend on $C$ up to proportionality. The series $n! \cdot \tau_n(u,E)$ is called quantum determinant and usually denoted by ${\rm qdet} \, T(u)$.
\end{rem}

The following result is classical (see e.g. \cite{molev2,nazol}):
\begin{prop} The coefficients of $\tau_k(u,C)$, $k=1,\ldots,n$ pairwise commute. The coefficients of ${\rm qdet} \, T(u)$ generate the center of $Y(\fgl_n)$.
\end{prop}

\begin{defn}
The commutative subalgebra generated by the coefficients of $\tau_k(u,C)$ is called {\itshape Bethe
subalgebra}. We denote it by $B(C)$.
\end{defn}

In \cite{nazol} Nazarov and Olshanski proved the following theorem:

\begin{thm}\label{MaxFreeTheorem}
Suppose that $C \in \End \mathbb{C}^n$ has pairwise distinct nonzero eigenvalues. Then the subalgebra $B(C)$ in $Y(\fgl_n)$ is maximal commutative. The coefficients of $\tau_1(u,C), \ldots, \tau_n(u,C)$ are free generators of $B(C)$.
\end{thm}

From now on we suppose that $C$ is diagonal and  $C_{aa} = \lambda_a,\ a=1,\ldots,n$. Let us write $\tau_k(u,C)$ explicitly in the generators $t_{ij}(u)$. Straightforward computation gives the following well known Lemma:
\begin{lem}
\label{vid}
We have
$$\tau_k(u,C) = \sum_{1 \leqslant a_1< \ldots < a_k \leqslant n} \lambda_{a_1} \ldots \lambda_{a_k} t_{a_1 \ldots a_k}^{a_1\ldots a_k}(u),$$
where $$t^{a_1 \ldots a_k}_{b_1 \ldots b_k} = \sum_{\sigma \in S_k} (-1)^\sigma \, \cdot \, t_{a_{\sigma(1)}b_1}(u) \ldots t_{a_{\sigma(k)}b_k}(u-k+1) = \sum_{\sigma \in S_k} (-1)^\sigma \, \cdot \, t_{a_1 b_{\sigma(1)}}(u-k+1) \ldots t_{a_k b_{\sigma(k)}}(u) $$ are the quantum minors.
\end{lem}
\begin{comment}\begin{proof}
Note that $\tau_k (u,C) = {\rm tr} \, A_k \cdot M_k$, where $M_k = C_1 T_1(u) \otimes C_2 T_2(u-1) \otimes \ldots \otimes C^{(k)} T_k(u-k+1).$
It is clear that in the basis $e_{i_1} \otimes \ldots \otimes e_{i_k}$ with the standard order we have $$M_{i_1 \ldots i_k; j_1 \ldots j_k} = \lambda_{i_1} \ldots \lambda_{i_k} t_{i_1 j_1}(u) \cdot t_{i_2 j_2}(u-1) \cdot \ldots \cdot t_{i_k j_k}(u-k+1).$$
Then in the same basis we have $$A_{i_1 \ldots i_k; j_1 \ldots j_k} = \sum_{\sigma \in S_k | j_p = i_{\sigma(p)}} (-1)^{\sigma}. $$
Note that if there exists $i_x = i_y$ then $(A_k \cdot M_k)_{i_1 \ldots i_k; j_1 \ldots j_k} = 0$. Suppose that $i_x \ne i_y$ for $x \ne y$.
Then
\begin{multline*} (A_k \cdot M_k)_{i_1 \ldots i_k; j_1 \ldots j_k} = \sum_{l_1, \ldots, l_k} A_{i_1 \ldots i_k; l_1 \ldots l_k} M_{l_1 \ldots l_k; j_1 \ldots j_k} = \sum_{\sigma \in S_k } (-1)^{\sigma} M_{i_{\sigma(1)} \ldots i_{\sigma(k)}; j_1 \ldots j_k} \end{multline*}
This implies that the trace is equal to
\begin{multline*}
\sum_{i_1 \ldots i_k} (A_k \cdot M_k)_{i_1 \ldots i_k; i_1 \ldots i_k} = \sum_{i_1 \ldots i_k} \sum_{\sigma \in S_k} (-1)^{\sigma} \lambda_{i_{\sigma(1)}} \ldots \lambda_{i_{\sigma(k)}} t_{i_{\sigma(1)} i_1}(u) \cdot \ldots \cdot t_{i_{\sigma(k)} i_k}(u-k+1) =   \\ = k! \sum_{i_1 < \ldots < i_k} \sum_{\sigma \in S_k} (-1)^{\sigma} \lambda_{i_1} \ldots \lambda_{i_k} t_{i_1 i_{\sigma(1)}}(u) \cdot \ldots \cdot t_{i_k i_{\sigma(k)}}(u-k+1)
\end{multline*}
The last equation follows from skew-symmetricity of quantum minors.
\end{proof}
\end{comment}

\subsection{Filtration on Yangian.}

We define an ascending filtration on $Y(\fgl_n)$ by setting the degree of the generators as
\begin{equation*}
\label{Degree}\deg t_{ij}^{(r)} = r.
\end{equation*}
Denote by $Y_r(\fgl_n)$ the vector subspace consisting of all elements with degrees not greater than $r$.
Consider the associated graded algebra
\label{filtration}
\begin{equation}
CY(\fgl_n) = \bigoplus_{r \geqslant 0} Y_r(\fgl_n) / Y_{r-1}(\fgl_n).
\end{equation}
There is the following analog of PBW-theorem for $Y(\fgl_n)$ (see \cite{molev}):
\begin{thm}
\label{PBW}$CY(\fgl_n)$ is the algebra of polynomials in the variables $\bar t_{ij}^{(r)}$, where $\bar t_{ij}^{(r)}$ is the image of $t_{ij}^{(r)}$ in $Y_r(\fgl_n) / Y_{r-1}(\fgl_n)$.
\end{thm}

The polynomial algebra $CY(\fgl_n)$ has a natural Poisson structure. For any two elements $x,y \in Y(\fgl_n)$ of degrees $p,q$ respectively the Poisson bracket of their images $\bar x,\bar y$ in $CY(\fgl_n)$ is
$$\{\bar x,\bar y\} = [x,y] \mod Y_{p+q-2}(\fgl_n).$$

\subsection{Some homomorphisms between the Yangians.}
Let us define two different embedding of $Y(\fgl_n)$ to $Y(\fgl_{n+k})$:
$$i_k: Y(\fgl_n) \to Y(\fgl_{n+k}); \ t^{(r)}_{ij} \mapsto t^{(r)}_{ij}$$
$$\varphi_k: Y(\fgl_n) \to Y(\fgl_{n+k}); \ t^{(r)}_{ij} \mapsto t^{(r)}_{k+i,k+j}$$
According to PBW-theorem these maps are injective.
Define a homomorphism
$$\pi_n: Y(\fgl_n) \to U(\fgl_n); \ t_{ij}(u) \mapsto \delta_{ij} + E_{ij}u^{-1}.$$ Here $E_{ij}$ are the standard generators of $\mathfrak{gl_n}$.
$\pi_n$ is a surjective homomorphism from $Y(\fgl_n)$ to $U(\fgl_n)$ known as \emph{evaluation} homomorphism.
By definition, put
$$\omega_n: Y(\fgl_n) \to Y(\fgl_n); \ T(u) \mapsto (T(-u-n))^{-1}.$$
It is well-known that $\omega_n$ is an involutive automorphism of $Y(\fgl_n)$.
We define a homomorphism
$$\psi_k = \omega_{n+k} \circ \varphi_k \circ \omega_n: Y(\fgl_n) \to Y(\fgl_{n+k}).$$
Note that $\psi_k$ is injective.

\begin{prop}\label{ipsi}\cite[Proposition 1.13.1]{molev2} $i_{k}(Y(\fgl_{n}))$ is centralized by $\psi_{n}(Y(\fgl_k))$ in $Y(\fgl_{n+k})$.
\end{prop}

We will need a bit more precise statement:

\begin{lem}
\label{commuteyang} The homomorphisms $i_{k}$ and $\psi_{n}$ define an embedding
$i_{k}\otimes\psi_{n}: Y(\fgl_{n}) \otimes Y(\fgl_k) \hookrightarrow Y(\fgl_{n+k})$.
\end{lem}
\begin{proof}
It suffices to show that the associated graded $\gr (i_{k}\otimes\psi_{n}): CY(\fgl_{n}) \otimes CY(\fgl_k) \to CY(\fgl_{n+k})$ is injective. But the latter is sends $t_{ij}^{(r)}\otimes1$ to $t_{ij}^{(r)}$ and $1\otimes t_{ij}^{(r)}$ to $t_{n+i,n+j}^{(r)}+L$, where $L$ is an expression of $t_{ij}^{(s)}$ with $s<r$. Hence it is clearly injective.
%we use \cite[Lemma 1.11.2]{molev2}. Indeed,  we see that the images of $t_{ij}(u)$ in the associated graded algebra contain different $t_{ij}^{(k)}$ and using PBW theorem we obtain a result.
\end{proof}

\subsection{Centralizer construction.}
Consider the map
\begin{equation}
\Phi_k: Y(\fgl_n) \to U(\fgl_{n+k}) \ \text{given by} \ \Phi_k = \pi_{n+k} \circ \omega_{n+k} \circ i_k.
\end{equation}
From ~\cite[Proposition 8.4.2]{molev2} it follows that ${\rm Im} \, \Phi_k \subset U(\fgl_{n+k})^{\fgl_k}$. Here we use an embedding $$\mathfrak{gl}_k \to \mathfrak{gl}_{n+k}, \ E_{ij} \to E_{i+n, j+n}.$$ 

Let $A_0 = \mathbb{C}[\mathcal{E}_1, \mathcal{E}_2, \ldots]$ be the polynomial algebra of infinite many variables. Define a grading on $A_0$ by setting $\deg \mathcal{E}_i = i$. For any $k$ we have a surjective homomorphism
$$z_k: A_0 \to Z(U(\fgl_{n+k}));\ \mathcal{E}_i \to \mathcal{E}_i^{(n+k)}.$$
where $\mathcal{E}_i^{(n+k)}$, $i=1,2,3,\ldots$ are the following generators of $Z(U(\fgl_{n+k}))$ of degree $i$, see \cite[Section 8.2]{molev2}:
$$ 1+\sum\limits_{i=1}^{n+k}\mathcal{E}_iu^{-i}=\pi_{n+k}({\rm qdet } T(U)).
$$

Consider the algebra $Y(\fgl_n) \otimes A_0$. This algebra has a well-defined ascending filtration given by
$$\deg a \otimes b = \deg a + \deg b.$$

For any $k \geqslant 0$ we define homomorphisms of filtered algebras
$$\eta_k: Y(\fgl_n) \otimes A_0 \to U(\fgl_{n+k})^{\fgl_k}; \ a \otimes b \to \Phi_k(a) \cdot z_k(b)$$

Denote by  $(Y(\fgl_n) \otimes A_0)_N$ the $N$-th filtered component, i.e. the vector space of the elements of degree not greater than $N$. From \cite[Theorem 8.4.3]{molev2} we have:

\begin{thm}
\label{qasym}
The sequence $\{\eta_k\}$ is an asymptotic isomorphism. This means that for any $N$ there exists $K$ such that for any $k>K$ the restriction of $\eta_k$ to the $N$-th filtered component $(Y(\fgl_n) \otimes A_0)_N$ is an isomorphism of vector spaces $(Y(\fgl_n) \otimes A_0)_N \simeq U(\fgl_{n+k})^{\fgl_k}_N$.
\end{thm}

According to Lemma~\ref{commuteyang}, there is a tensor product of two commuting Yangians $i_{m}(Y(\fgl_{n-m}))$ and $\psi_{n-m}(Y(\fgl_{m}))$ inside $Y(\fgl_n)$.  Let us see what happens with these two commuting Yangians under the centralizer construction map $\eta_{k}:Y(\fgl_n)\otimes A_0\to U(\fgl_{n+k})^{\fgl_k}$.  
The algebra $U(\fgl_{n+k})$ is generated by $E_{ij}, 1 \leqslant i,j \leqslant n+k$; $\fgl_k$ is generated by $E_{ij}, n+1 \leqslant i,j \leqslant n+k$. The subalgebra $U(\fgl_{m+k})\subset U(\fgl_{n+k})$ is generated by $E_{ij}, n-m+1 \leqslant i,j \leqslant n+k$. Hence $U(\fgl_{m+k})^{\fgl_k}$ and $U(\fgl_{n+k})^{\fgl_{m+k}}$ are naturally subalgebras in $U(\fgl_{n+k})^{\fgl_k}$.
\begin{lem}
The restriction of $\eta_k$ to $i_m(Y(\fgl_{n-m}))$ is $\eta_{k+m}:Y(\fgl_{n-m})\to U(\fgl_{n+k})^{\fgl_{k+m}}$.
The restriction of $\eta_k$ to $\psi_{n-m}(Y(\fgl_m))$ is $\eta_{k}:Y(\fgl_m)\to U(\fgl_{m+k})^{\fgl_k}$.
\end{lem}
\begin{proof}
The first statement is immediate from the following commutative diagram:
\begin{center}
\[ \begin{diagram}
\node{Y(\fgl_{n-m})} \arrow[2]{e,t}{i_{m}}
\arrow{ese,b}{i_{k+m}}
\node[2]{Y(\fgl_n)} \arrow{s,r}{i_k}\\
\node[3]{Y(\fgl_{n+k})}
\end{diagram}\]
\end{center}
To prove second statement, consider another diagram:
\begin{center}
\[ \begin{diagram}
\node{Y(\fgl_m)}
\arrow{e,t}{\psi_{n-m}}
\arrow{s,r}{i_k}
\node{Y(\fgl_n)}
\arrow{s,r}{i_k} \\
\node{Y(\fgl_{m+k})} \arrow{e,t,..}{\psi_{n-m}}
\arrow{s,r}{\omega_{m+k}}
\node{Y(\fgl_{n+k})} \arrow{s,r}{\omega_{n+k}}\\
\node{Y(\fgl_{m+k})} \arrow{e,t}{\varphi_{n-m}}
\node{Y(\fgl_{n+k})}
\end{diagram}\]
\end{center}

The lower square is commutative by definition so it is enough to prove that upper square is commutative as well. But it follows from \cite[Lemma 1.11.2]{molev2}.
\end{proof}

\subsection{Definition of limit subalgebras.}\label{def-lim}
It is possible to construct new commutative subalgebras as limits of Bethe subalgebras.
We describe here what the "limit" means.
Let $T^{reg}$ be the set of regular elements of the maximal torus $T$ of $GL_n$. Let $C$ be an element of $T^{reg}$. Consider $B_r(C):= Y_r(\fgl_n) \cap B(C)$. In \cite{nazol} it was proved that the images of the coefficients of $\tau_1(u,C), \ldots, \tau_n(u,C)$ freely generate the subalgebra $\bar B(C) = \gr B(C)\subset CY(\fgl_n)$. Hence the dimension $d(r)$ of $B_r(C)$ does not depend on $C$.
Therefore for any $r \geqslant 1$ we have a map $\theta_r$ from $T^{reg}$ to $\bigtimes_{i=1}^r {\rm Gr}(d(i),\dim Y_{i}(\fgl_n))$ such that $C \to (B_1(C), \ldots, B_r(C))$.
Denote the closure of $\theta_r(T^{reg})$ (with respect to Zariski topology) by $Z_r$. There are well-defined projections  $\rho_r: Z_r \to Z_{r-1}$ for all $r \geqslant 1$. The inverse limit $Z = \varprojlim Z_r$ is well-defined as a pro-algebraic scheme and is naturally a parameter space for some family of commutative subalgebras which extends the family of Bethe subalgebras.

Indeed, any point $X\in Z$ is a sequence $\{x_r\}_{r \in \mathbb{N}}$ where $x_r \in Z_r$ such that $\rho_r(x_r) = x_{r-1}$. Every $x_r$ is a point in $\bigtimes_{i=1}^r {\rm Gr}(d(i),\dim Y_{i}(\fgl_n))$ i.e. a collection of subspaces $B_{r,i}(X)\subset Y_{i}(\fgl_n)$ such that $B_{r,i}(X)\subset B_{r,i+1}(X)$ for all $i<r$. Since $\rho_r(x_r) = x_{r-1}$ we have $B_{r,i}(X)= B_{r-1,i}(X)$ for all $i<r$. Let us define the subalgebra corresponding to $X\in Z$ as $B(X):= \bigcup_{r=1}^{\infty} B_{r,r}(X)$. We claim that this subalgebra is commutative and we call it {\em limit subalgebra}. Indeed, $B(X)$ is a commutative subalgebra because being a commutative subalgebra is a Zariski-closed condition: we have $B_r(C)\cdot B_s(C)\subset B_{r+s}(C)$ for all $C\in T^{reg}$ for all $r,s$ and this product is commutative, hence we get the same for the product $B_{r,r}(X)\cdot B_{s,s}(X)=B_{r+s,r}(X)\cdot B_{r+s,s}(X)\subset B_{r+s,r+s}(X)$. Note that this subalgebra has the same Poincare series as $B(C)$.

\begin{rem}
In \cite{shuvalov} the limits shift of argument subalgebras are defined in the analytic topology, just as limits of $1$-parametric families of generic shift of argument subalgebras. But it is well-known that the closure of an affine algebraic variety in a complex projective space with respect to Zariski topology coincides with its closure with respect to the analytic topology. So the closure of the parameter space is the same for both definitions of the limit. In our work we use both approaches.
\end{rem}

\section{Moduli spaces of stable rational curves}

\subsection{The space $\overline{M_{0,n+1}}$.} Let $\overline{M_{0,n+1}}$ denote the Deligne-Mumford space of stable rational curves with $n+1$ marked points. The points of $\overline{M_{0,n+1}}$ are isomorphism classes of curves of genus $0$, with $n+1$ ordered marked points and possibly with nodes, such that each component has at least $3$ distinguished points (either marked points or nodes). One can represent the combinatorial type of such a curve as a tree with $n+1$ leaves with inner vertices representing irreducible components of the corresponding curve, inner edges corresponding to the nodes and the leaves corresponding to the marked points. Informally, the topology of $\overline{M_{0,n+1}}$ is determined by the following rule: when some $k$ of the distinguished points (marked or nodes) of the same component collide, they form a new component with $k+1$ distinguished points (the new one is the intersection point with the old component). In particular, the tree describing the combinatorial type of the less degenerate curve is obtained from the tree corresponding to the more degenerate one by contracting an edge.

The space $\overline{M_{0,n+1}}$ is a smooth algebraic variety. It can be regarded as a compactification of the configuration space $M_{0,n+1}$ of ordered $(n+1)$-tuples $(z_1,z_2,\ldots,z_{n+1})$ of pairwise distinct points on $\BC\BP^1$ modulo the automorphism group $PGL_2(\BC)$. Since the group $PGL_2(\BC)$ acts transitively on triples of distinct points, we can fix the $(n+1)$-th point to be $\infty\in\BC\BP^1$ and fix the sum of coordinates of other points to be zero. Then the space $M_{0,n+1}$ gets identified with the quotient ${\rm Conf}_n / \BC^*$ where ${\rm Conf}_n:=\{(z_1,\ldots,z_n)\in\BC^n\ |\ z_i\ne z_j,\ \sum\limits_{i=1}^n z_i=0\}$, and the group $\BC^*$ acts by dilations. Under this identification of $M_{0,n+1}$, the space $\overline{M_{0,n+1}}$ is just the GIT quotient of the iterated blow-up of the subspaces of the form $\{z_{i_1}=z_{i_2}=\ldots=z_{i_k}\}$ in $\BC^{n-1}$ by the natural $\BC^*$ action by dilations.

\subsection{Stratification and operad structure on $\overline{M_{0,n+1}}$.}

The space $\overline{M_{0,n+1}}$ is stratified as follows. The strata are indexed by the combinatorial types of stable rational curves, i.e. by rooted trees with $n$ leaves colored by the marked points $z_1,\ldots,z_n$ (the root is colored by $z_{n+1}=\infty$). The stratum corresponding to a tree $T$ lies in the closure of the one corresponding to a tree $T'$ if and only if $T'$ is obtained from $T$ by contracting some edges.

The spaces $\ol{M_{0,n+1}}$ form a topological operad. This means that one can regard each point of the space $\ol{M_{0,n+1}}$ as an $n$-ary operation with the inputs at marked points $z_1,\ldots,z_n$ and the output at $\infty$. Then one can substitute any operation of this form to each of the inputs. More precisely, for any partition of the set $\{1,\ldots,n\}$ into the disjoint union of subsets $M_1,\ldots,M_k$ with $|M_i|=m_i\ge1$ there is a natural substitution map $\gamma_{k;M_1,\ldots,M_k}:\ol{M_{0,k+1}}\times\prod\limits_{i=1}^k\ol{M_{0,m_i+1}}\to\ol{M_{0,n+1}}$ which attaches the $i$-th curve $C_i\in\ol{M_{0,m_i+1}}$ to the $i$-th marked point of the curve $C_0\in\ol{M_{0,k+1}}$ by gluing the $m_{i}+1$-th marked point of each $C_i$ with the $i$-th marked point of $C_0$. In fact all substitution maps $\gamma_{k;M_1,\ldots,M_k}$ are compositions of the elementary ones with $m_1=\ldots=m_{k-1}=1$.

The compositions of the substitution maps are indexed by rooted trees describing the combinatorial type of the (generic) resulting curves. In particular, each stratum of $\ol{M_{0,n+1}}$ is just the image of the product of the open strata of appropriate $\prod\ol{M_{0,m+1}}$ under some composition of substitution maps. In particular, strata of codimension $1$ are just the images of the open strata of $\ol{M_{0,k+1}}\times\ol{M_{0,n-k+1}}$ under the elementary substitution maps and can be obtained as the limit set when the points $z_1,\ldots,z_k$ collide and other points stay isolated. Since each substitution map is a composition of the elementary ones, generic point of each stratum can be obtained iterating this limiting procedure.

\section{Shift of argument subalgebras.}

\subsection{The algebras $F(C)$.} Let $\fg$ be a reductive Lie algebra. To any $C\in\fg^*$ one can assign a Poisson-commutative subalgebra in $S(\fg)$ with respect to the standard Poisson bracket (coming from the universal enveloping algebra $U(\fg)$ by the PBW theorem).
Let $ZS(\fg)=S(\fg)^{\fg}$ be the center of $S(\fg)$ with
respect to the Poisson bracket. The algebra $F(C)\subset S(\fg)$
generated by the elements $\partial_{C}^n\Phi$, where $\Phi\in
ZS(\fg)$, (or, equivalently, generated by central elements of
$S(\fg)=\BC[\fg^*]$ shifted by $tC$ for all $t\in\BC$) is
Poisson-commutative and has maximal possible transcendence degree. More precisely, we have the following

\begin{thm}\label{mf} ~\cite{mf} For regular semisimple $C\in\fg$ the algebra $F(C)$
is a free commutative subalgebra in $S(\fg)$ with
$\frac{1}{2}(\dim\fg+\rk\fg)$ generators (this means that
$F(C)$ is a commutative subalgebra of maximal possible
transcendence degree). One can take the elements
$\partial_{C}^n\Phi_k$, $k=1,\dots,\rk\fg$,
$n=0,1,\dots,\deg\Phi_k$, where $\Phi_k$ are basic
$\fg$-invariants in $S(\fg)$, as free generators of $F(C)$.
\end{thm}

\begin{rem} In particular for $\fg=\fgl_n$, the algebra $F(C)$ has $n-k$ generators of degree $k$ for each $k=1,\ldots,n$.
\end{rem}

Let $\fh\subset\fg$ be a Cartan subalgebra of the Lie algebra
$\fg$. Denote by $\Delta_+$ the set of positive roots of
$\fg$.  We assume that $C$ is a regular semisimple
element of the Cartan subalgebra $\fh\subset\fg=\fg^*$. The
linear and quadratic part of the subalgebras $F(C)$
can be described as follows (see \cite{vinberg}):
\begin{gather*}F(C)\cap\fg=\fh,\\ F(C)\cap S^2(\fg)=S^2(\fh)\oplus
Q_C,\ \text{where}\
Q_C=\{\sum\limits_{\alpha\in\Delta_+}\frac{\langle\alpha,h\rangle}{\langle\alpha,C\rangle}e_{\alpha}e_{-\alpha}|h\in\fh\}.\end{gather*}

The results of Vinberg \cite{vinberg} and Shuvalov \cite{shuvalov} imply that the limit subalgebra of the family $F(C)$ is uniquely determined by its quadratic component. Hence the variety parameterizing limit shift of argument subalgebras in $S(\fg)$ is just the closure of the family $Q_C$ in the Grassmannian $Gr(\rk\fg, \dim S^2(\fg))$. In \cite{shuvalov} Shuvalov described the closure of the family of
subalgebras $F(C)\subset S(\fg)$ under the condition
$C\in \fh^{reg}$ (i.e., for regular $C$ in the fixed Cartan
subalgebra). In particular, the following holds:

\begin{thm}\label{shuvalov}\cite{shuvalov} Suppose that $C(t)=C_0+tC_1+t^2C_2+\dots\in \fh^{reg}$
for generic $t$. Let $\fz_{\fg}(C_i)$ be the centralizer of $C_i$ in $\fg$.  Set $\fz_k=\bigcap\limits_{i=0}^k\fz_{\fg}(C_i)$,
$\fz_{-1}=\fg$. Then we have
\begin{enumerate}
\item the subalgebra $\lim\limits_{t\to0}F(C(t))\subset S(\fg)$
is generated by elements of $S(\fz_k)^{\fz_k}$ and their
derivatives of any order along $C_{k+1}$ for all $k=-1,0,\ldots$. \item
$\lim\limits_{t\to0}F(C(t))$ is a free commutative algebra. %One can take the generators of $S(\fz_k')^{\fz_k}$ for all $k$ and some of their
%derivatives along $C_{k+1}$ as a set of free generators of the limit subalgebra $\lim\limits_{t\to0}F(C(t))$.
\end{enumerate}
\end{thm}

Moreover, according to the results of A.~Tarasov \cite{taras} the subalgebras dicussed above (both $F(C)$ and the limit ones) are maximal Poisson-commutative subalgebras in $S(\fg)$ (i.e. coincide with their Poisson centralizers).

\subsection{Lifting to $U(\fg)$.} In \cite{ryb4} the subalgebras $F(C)$ were quantized, i.e. the existence of such commutative subalgebras $\hat F(C)\subset U(\fg)$ that $\gr \hat F(C) = F(C)$ was proved. In \cite{ryb} it is proved that $F(C)$ is the Poisson centralizer of the space $Q_C$ for generic $C$ hence the lifting of $F(C)$ to $U(\fg)$ is unique for generic $C$. In the case $\fg=\fgl_n$ we have a stronger statement due to A.~Tarasov \cite{taras2}: any subalgebra from this family (including the limit ones) can be uniquely lifted to the universal enveloping algebra and, moreover, there is a particular choice of the generators such that this lifting is just the symmetrization map on the generators. In particular, for $\fg=\fgl_n$, the varieties parameterizing limit subalgebras of the families $F(C)$ and $\hat F(C)$ are the same. In fact, we expect that the latter is true for arbitrary $\fg$.

In the case $\fg=\fgl_n$ the subalgebra $\hat F(C)\subset U(\fgl_n)$ is in fact the image of $B(C)\subset Y(\fgl_n)$ under the evaluation homomorphism $Y(\fgl_n)\to U(\fgl_n)$ (we can assume that $C$ is nondegenerate since $\hat F(C)$ does not change after adding a scalar matrix to $C$). This fact plays crucial role in what follows. Unfortunately this does not generalize to arbitrary $\fg$ since there is no evaluation homomorphism in general.

\subsection{$\ol{M_{0,n+1}}$ and shift of argument subalgebras.} 
In \cite{shuvalov} the explicit description of limit subalgebras is given together with a set-theoretical description of the parameter space. The results of Shuvalov in the $\fgl_n$ case can be reformulated in the following way.

The shift of argument subalgebras in $S(\fgl_n)$ depend on a parameter $C\in T^{reg}$ and do not change under the transformations $C\mapsto aC+bE$. So the parameter space for the subalgebras $F(C)\subset S(\fgl_n)$ can be regarded as $M_{0,n+1}$.

\begin{prop}\label{prop-fclosure} The closure of the parameter space for the shift of argument subalgebras in $S(\fgl_n)$ is $\ol{M_{0,n+1}}$. The same is true for the parameter space for the family $\hat F(C)$.
\end{prop}

\begin{proof} For $\fg=\fgl_n$ the regular Cartan element $C$ is a diagonal matrix with pairwise distinct eigenvalues $z_1,\ldots,z_n$. From the results of Shuvalov \cite{shuvalov} and Vinberg \cite{vinberg} it follows that any limit subalgebra is uniquely determined by its quadratic graded component. According to the results of Tarasov \cite{taras2}, the lifting of any limit shift of argument subalgebra is unique hence uniquely determined by the quadratic component as well. Note that in this case $Q_C$ is the linear span of the quadratic elements $H_i:=\sum\limits_{j\ne i}\frac{e_{ij}e_{ji}}{z_i-z_j}$ which are the coefficients of (an appropriate version of) the $KZ$ connection. In particular the space $Q_C$ does not change under simultaneous affine transformations of the $z$'s, hence the space of parameters is naturally the configuration space of $n$ pairwise distinct points on the affine line or, equivalently, the configuration space $M_{0,n+1}$ of $n+1$ pairwise distinct points on the projective line. In \cite{AFV} Aguirre, Felder and Veselov considered the same subspaces $Q_C$ universally (i.e. as subspaces in the Drinfeld-Kohno holonomy Lie algebra) and showed that the closure of the family $Q_C$ is the Deligne-Mumford compactification $\overline{M_{0,n+1}}$. This means that the parameter space for limit shift of argument subalgebras (both classical and quantum) for $\fg=\fgl_n$ is $\overline{M_{0,n+1}}$.
\end{proof}

One can define the subalgebra $F(X)\subset S(\fgl_n)$ corresponding to a degenerate curve $X\in \ol{M_{0,n+1}}$ recursively as follows. Let $X_\infty$ be the irreducible component of $X$ containing the marked point $\infty$. To any distinguished point $\lambda\in X_\infty$ we assign the number $k_\lambda$ of marked points on the (reducible) curve $X_\lambda$ attached to $X_\infty$ at $\lambda$. Let $C$ be the diagonal matrix with the eigenvalues $\lambda$ of multiplicity $k_\lambda$ for all distinguished points $\lambda\in X_\infty$. Then the corresponding shift of argument subalgebra $F(C)\subset S(\fgl_n)$ is centralized by the Lie subalgebra $\bigoplus\limits_{\lambda}\fgl_{k_\lambda}$ in $\fgl_n$ and contains the Poisson center $S(\bigoplus\limits_\lambda \fgl_{k_\lambda})^{\bigoplus\limits_\lambda \fgl_{k_\lambda}}$. The subalgebra corresponding to the curve $X$ is just the product of $F(C)\subset S(\fgl_n)$ and the subalgebras corresponding to $X_\lambda$ in $S(\fgl_{k_\lambda})\subset S(\fgl_n)$ for all distinguished points $\lambda\in X_\infty$ (for this we need to define the point $\infty$ on each $X_\lambda$ -- it is just the intersection with $X_\infty$).
Having a unique lifting of each $F(C)$ to the universal enveloping algebra $U(\fgl_n)$ we can do the same and attach a commutative subalgebra in $U(\fgl_n)$ to any stable rational curve $X\in\ol{M_{0,n+1}}$. Since every limit shift of argument subalgebra has a unique lifting to $U(\fgl_n)$, every limit subalgebra of the family $\hat F(C)\subset U(\fgl_n)$ is of this form. Now we can make a precise statement:

\begin{prop}
\label{tensorprod}
The subalgebra $F(X)$ corresponding to a degenerate curve $X$ is the tensor product $F(C)\otimes_{S(\bigoplus\limits_\lambda \fgl_{k_\lambda})^{\bigoplus\limits_\lambda \fgl_{k_\lambda}}}\bigotimes\limits_\lambda F(X_\lambda)$. The same is true for the lifting $\hat F(X)\subset U(\fgl_n)$.
\end{prop}

\begin{proof}
In order to prove this proposition we need the following Lemma:
\begin{lem}
\label{tensor} Suppose that $\fg$ is a reductive Lie algebra, $\fg_0$ -- reductive subalgebra of $\fg$. Then the subalgebras $U(\fg)^{\fg_0}$ and $U(\fg_0)$ in $U(\fg)$ are both free $U(\fg_0)^{\fg_0}$-modules. Moreover, the product of these subalgebras in $U(\fg)$ is:
$$U(\fg)^{\fg_0} \cdot U(\fg_0) \simeq U(\fg)^{\fg_0} \otimes_{U(\fg_0)^{\fg_0}} U(\fg_0).$$
The same holds for the associated graded algebras. Namely, $S(\fg)^{\fg_0}$ and $S(\fg_0)$ are free $S(\fg_0)^{\fg_0}$-modules, and
$$S(\fg)^{\fg_0} \cdot S(\fg_0) \simeq S(\fg)^{\fg_0} \otimes_{S(\fg_0)^{\fg_0}} S(\fg_0).$$
\end{lem}
\begin{proof} We can assume without loss of generality that $\fg_0$ does not contain nontrivial ideals of $\fg$. Then this is a particular case of Knop's theorem on Harish-Chandra map for reductive group actions (see \cite{Kn}, Theorem 10.1 and items (d) and (e) of the Main Theorem). Indeed, the Main Theorem of \cite{Kn} states that for any reductive $H$ and any smooth affine $H$-variety $X$ the algebra $D(X)^H$ of $H$-invariant differential operators and its commutant $U(X)$ in the algebra $D(X)$ are both free modules over the center of $D(X)^H$. Moreover, the product $D(X)^H\cdot U(X)\subset D(X)$ is the tensor product $D(X)^H\otimes_{ZD(X)^H} U(X)$ of $D(X)^H$ and $U(X)$ over the center of $D(X)^H$. To get the desired statement we just apply this to the $H=G\times G_0$-action on $X=G$, where $G$ acts from the right and $G_0$ acts on from the left. For this case we have $D(X)^{G\times G_0}$ is $U(\fg)^{\fg_0}$ and its commutant in $D(X)$ is $U(X)=U(\fg_0)\otimes_\BC U(\fg)$ (generated by momenta of the left $G_0$-action and the right $G$-action). By Theorem 10.1 of \cite{Kn}, the center of $U(\fg)^{\fg_0}$ is $U(\fg_0)^{\fg_0}\otimes_\BC U(\fg)^{\fg}$ and the algebra $U(X)$ contains $U(\fg_0)\otimes_\BC U(\fg)^\fg$ as the subalgebra of (right) $G$-invariants. Hence by the Main Theorem of \cite{Kn} we have the desired assertion.
Theorems 9.4 and 9.8 of \cite{Kn} imply the same for the associated graded algebras (in fact the Main Theorem of \cite{Kn} is a consequence of the same fact for the associated graded algebras).
\end{proof}

\begin{comment}
\begin{cor}
\label{corinter}
Let $A=\lim\limits_{t\to0} F(C(t))$ where $C(t)=C_0+t \cdot C_1\in \fh_{reg}$. Then we have $A=F(C_0)\otimes_{ZU(\fz_\fg(C_0))}F(C_1)$ where $F(C_1)$ is the shift of argument subalgebra in $U(\fz_\fg(C_0))$. In particular $A\cap U(\fg)^{\fz_\fg(C_0)} = F(C_0)$ and $A\cap U(\fz_\fg(C_0)) = F(C_1)\subset U(\fz_\fg(C_0))$. The same is true for the associated graded subalgebras in $S(\fg)$.
\end{cor}
\end{comment}

Now let us describe the simplest limit subalgebras corresponding to the case when all the curves $X_\lambda$ are irreducible.

\begin{lem}
\label{size}
Let $C_0=\diag(\underbrace{\lambda_1, \ldots, \lambda_1}_{k_1}, \ldots, \underbrace{\lambda_l, \ldots, \lambda_l}_{k_l})$ and $C_i=\diag(\underbrace{\mu_{i,1}, \ldots, \mu_{i,k_i}}_{k_i})$ for $i=1,\ldots,l$ such that $\lambda_r\ne\lambda_s$ and $\mu_{i,r}\ne\mu_{i,s}$ for $r\ne s$. Then the element
\begin{eqnarray*}C(t) := C_0 + t \cdot \diag(C_1, \ldots, C_l)\end{eqnarray*}
is regular (as an element of the \emph{Lie algebra} $\fgl_n$) for small $t$.
The subalgebra $\lim_{t \to 0} F(C(t))$ is the tensor product $$F(C_0) \otimes_{Z(S(\fgl_{k_1}\oplus \ldots \oplus \fgl_{k_l}))} (F(C_1)\otimes\ldots\otimes F(C_l)).$$ 
Here $F(C_i)$ is a subalgebra in $S(\fgl_{k_i})\subset S(\fgl_{k_1}\oplus \ldots \oplus \fgl_{k_l})$.
Moreover, $$\lim_{t \to 0} F(C(t)) \cap S(\fgl_{n})^{\fgl_{k_1}\oplus \ldots \oplus \fgl_{k_l}} = F(C_0).$$
Here $n=k_1+ \ldots +k_l$.
The same holds for quantum shift of argument subalgebras $\hat F(C)$.
\end{lem}

\begin{proof}
From \cite{shuvalov}
 it follows that $\lim_{t \to 0} = F(C_0) \cdot (F(C_1)\otimes\ldots\otimes F(C_l))$.
We know that $F(C_0) \subset S(\fgl_n)^{\fgl_{k_1} \oplus \ldots \oplus \fgl_{k_l}}$ (because $[\partial_C, {\rm ad} \, x] = \partial_{[C,x]}$) and $(F(C_1)\otimes\ldots\otimes F(C_l)) \subset S(\fgl_{k_1} \oplus \ldots \oplus \fgl_{k_l})$.
Moreover, from proof of \cite[Lemma 1]{shuvalov}  it follows that $F(C_0) \cap (F(C_1)\otimes\ldots\otimes F(C_l)) = Z(S(\fgl_{k_1}\oplus \ldots \oplus \fgl_{k_l}))$. Using Lemma \ref{tensor} we obtain the result.
\end{proof}

Since the closure of the parameter space for subalgebras $F(C)$ is $\overline{M_{0,n+1}}$, any limit subalgebra can be obtained by iterating the above degeneration. Iterating limits from Lemma above we obtain Proposition ~\ref{tensorprod}.
\end{proof}

Lemma~\ref{size} implies the following formula for the Poincare series of $F(C)$ for any diagonal $C$. Denote by $Z_k(x)$ the formal power series $\prod\limits_{i=1}^k(1-x^i)^{-1}$. We set $P_k(x):=\prod\limits_{i=1}^k Z_i(x)$. Note that $P_n(x)$ is the Poincare series of $F(C)$ for regular $C$.

\begin{lem}\label{fsize} Let $C=\diag(\underbrace{\lambda_1, \ldots, \lambda_1}_{k_1}, \ldots, \underbrace{\lambda_l, \ldots, \lambda_l}_{k_l})$ such that $\lambda_r\ne\lambda_s$. Then $F(C)$ is a free polynomial algebra with the Poincare series $$ P_n(x)\prod\limits_{i=1}^{l}\frac{Z_{k_i}(x)}{P_{k_i}(x)}.
$$
\end{lem}
\begin{proof} Straightforward from Lemma~\ref{size}.
\end{proof}

\begin{lem}
\label{l36}
Let $\mathfrak{A}$ be a family of subalgebras of the form $F(\diag(C,\underbrace{0, \ldots, 0}_{k})), C \in T^{reg} \subset GL_n$. Then \\
1) Every limit subalgebra of the family $\mathfrak{A}$ is a maximal commutative subalgebra of $S(\fgl_{n+k})^{\fgl_k}$. \\
2) Every limit subalgebra of this family is a free polynomial algebra with the following Poincare series:
$$P(x) = \prod \dfrac{1}{(1-x)^{n+1}} \cdot \dfrac{1}{(1-x^2)^{n+1}} \cdot \ldots \cdot \dfrac{1}{(1-x^k)^{n+1}} \cdot \dfrac{1}{(1-x^{k+1})^{n}} \cdot \dfrac{1}{(1-x^{k+2})^{n-1}} \cdot \ldots \cdot \dfrac{1}{(1-x^{n+k})}.$$
3) The closure of the parameter space for the family $\mathfrak{A}$ is $\ol{M_{0,n+2}}$.
\end{lem}
\begin{proof}
Let $a(t)$ belongs to $T^{reg}$ for small $t$ and consider
$b(t) = \diag(a(t),\underbrace{0, \ldots, 0}_{k})$.
Let $D = (\underbrace{0, \ldots, 0}_n, \mu_1 \ldots \mu_k)$ and degree $N$ such that $N > \deg a(t)$ and $c(t) = b(t) + t^N D$ is regular for small $t$.
Then
\begin{equation}
\label{f1}
\lim_{t \to 0} F(c(t)) =  \lim_{t \to 0} F(b(t)) \otimes_{Z(S(\fgl_k))} F(D)
\end{equation} is free and maximal commutative.

Suppose that $\lim_{t \to 0} F(b(t))$ is not maximal. This implies that there exists an element $x \in S(\fgl_{n+k})^{\fgl_k}$ such that $\{x,\lim_{t \to 0}F(b(t))\} = 0$.
This means that $x$ commutes with every element from the limit subalgebra (\ref{f1}). But the limit subalgebra is maximal therefore $x \in \lim_{t \to 0} F(c(t))\cap S(\fgl_{n+k})^{\fgl_k}$ which is $\lim_{t \to 0} F(b(t))$ by Lemma ~\ref{size}. Thus $\lim_{t \to 0} F(b(t))$ is maximal.
\begin{comment}
Further, from~\cite{shuvalov} follows that $\lim_{t \to 0}  F(a(t))$ has the following system of free generators:
some free generators of  $\lim_{t \to 0}F(C_1(t))$, some free generators of $F(C_2)$ and generators of $S(\fgl_k)^{\fgl_k}$. So, first and third set are free generators of $\lim_{t \to 0} F(C_1(t))$.
\end{comment}

The second assertion follows from the fact that Poincare series of limit subalgebra \\ $\lim_{t \to 0} F(b(t))$ is the same as for algebra $F(b(t))$ for generic $t$ and from Lemma~\ref{fsize}.
To prove third assertion  let fix $C_1 \in \fgl_k$ and consider a family of subalgebras of the form $$F(C_0) \otimes_{Z(S(\fgl_k))} F(C_1),$$ where $C_0 = \diag(C,\underbrace{0, \ldots, 0}_{k}), C \in \fh^{reg}$. This family is parameterized by  \\ $\gamma_{n+1;\{1\},\{2\},\ldots,\{n\},\{n+1,\ldots,n+k\}}(M_{0,n+2} \times \{pt\}) \subset \gamma_{n+1;\{1\},\{2\},\ldots,\{n\},\{n+1,\ldots,n+k\}}(M_{0,n+2} \times  M_{0,k+1}) \subset M_{0, n+k+1}$. But  the closure of $M_{0, n+2} \times \{pt\}$ in $\ol{M_{0,n+k+1}}$ is $\ol{M_{0,n+2}} \times \{pt\}$ therefore the closure of the parameter space for the family $\mathfrak{A}$ is $\ol{M_{0,n+2}}$.
\end{proof}

\begin{cor} The same results (Lemmas \ref{size}, \ref{fsize} and \ref{l36}) are true for the liftings of shift of argument subalgebras to the universal enveloping algebra $U(\fgl_n)$.
\end{cor}
\begin{proof}
Follows from the fact that shift of argument subalgebras admits unique lifting to $U(\fgl_n)$ \\ ~\cite{taras2}.
\end{proof}

\section{The results.}
The subalgebra $B(C)$ does not change under dilations of $C$, so the space of parameters for the family of Bethe subalgebras in the Yangian $Y(\fgl_n)$ is $T^{reg}/\mathbb{C}^*$. Let $Z$ be the pro-algebraic scheme defined in Section~\ref{def-lim} (note that $Z$ is naturally a closure of $T^{reg}/\mathbb{C}^*$). By analogy with the shift of argument subalgebras, we can regard $T^{reg}/\mathbb{C}^*$ as the moduli space $M_{0,n+2}$, i.e. the space of rational curves with $n+2$ marked points. More precisely, to any matrix $C$ with the eigenvalues $z_1,\ldots,z_n$, we can assign a rational curve $\mathbb{P}^1$ with the marked points $0,z_1,\ldots,z_n,\infty$. Our main result is the following

\begin{thm}
\label{result}
\begin{enumerate}
\item $Z$ is a smooth algebraic variety isomorphic to $\overline{M_{0,n+2}}$.
\item For any point $X\in\overline{M_{0,n+2}}$, the corresponding commutative subalgebra $B(X)$ in $Y(\fgl_n)$ is free and maximal commutative.
\end{enumerate}
\end{thm}

Next, the operad structure on the spaces $\ol{M_{0,n+1}}$ leads to a recursive description of the limit subalgebras analogous to that of the limit shift of argument subalgebras. Let $X_\infty$ be the irreducible component of $X\in\ol{M_{0,n+2}}$ containing the marked point $\infty$. We identify $X_\infty$ with the standard $\mathbb{CP}^1$ in such a way that the marked point $\infty$ is $\infty$ and the point where the curve containing the marked point $0$ touches $X_\infty$ is $0$. To any distinguished point $\lambda\in X_\infty$ we assign the number $k_\lambda$ of marked points on the maximal (possibly reducible) curve $X_\lambda$ attached to $X_\infty$ at $\lambda$ (we set $k_\lambda=1$ if $X_\lambda$ is a (automatically, marked) point). Let $C$ be the diagonal $(n-k_0)\times(n-k_0)$-matrix with the eigenvalues $\lambda$ of multiplicity $k_\lambda$ for all distinguished points $0\ne\lambda\in X_\infty$. Then the subalgebra $i_{k_0}(B(C))$ centralized by the Lie subalgebra $\bigoplus\limits_{\lambda\ne0}\fgl_{k_\lambda}$ in $\fgl_{n-k_0}\subset i_{k_0}(Y(\fgl_{n-k_0}))\subset Y(\fgl_n)$ and by the complement sub-Yangian $\psi_{n-k_0}(Y(\fgl_{k_0}))$.
The subalgebras corresponding to boundary points of $\ol{M_{0,n+2}}$ can be inductively described as follows:

\begin{thm}\label{result1} \begin{enumerate} \item The limit Bethe subalgebra corresponding to the curve $X\in\ol{M_{0,n+2}}$ is the product of the following $3$ commuting subalgebras: first, $i_{k_0}(B(C))\subset i_{k_0}(Y(\fgl_{n-k_0}))\subset Y(\fgl_n)$, second, the subalgebra corresponding to $X_0$ in the complement sub-Yangian $\psi_{n-k_0}(Y(\fgl_{k_0}))\subset Y(\fgl_n)$ and third, the limit shift of argument subalgebras $\hat F(X_\lambda)$ in $U(\fgl_{k_\lambda})\subset i_{k_0}(Y(\fgl_{n-k_0}))\subset Y(\fgl_n)$ for all distinguished points $\lambda\ne0$ (again we define the point $\infty$ on each $X_\lambda$ just as the intersection with $X_\infty$).
\item $i_{k_0}(B(C))$ contains the center of every $U(\fgl_{k_\lambda})\subset i_{k_0}(Y(\fgl_{n-k_0}))$. The above product is in fact the tensor product $$\psi_{n-k_0}(B(X_0))\otimes_{\BC}i_{k_0}(B(C))\otimes_{ZU(\bigoplus_{\lambda\ne0} \fgl_{k_\lambda})}\bigotimes_{\lambda\ne0}\hat F(X_\lambda).$$
\end{enumerate}
\end{thm}

\begin{comment}
\begin{rem}
$\hat F(X_\lambda)$ is the limit quantum shift of argument subalgebra defined in Proposition ~\ref{tensorprod}.
\end{rem}
\end{comment}

\begin{thm}
\label{result2}
1) Let $C_0=\diag(\lambda_1, \ldots, \lambda_{n-k})$ and $C_1=\diag(\mu_1, \ldots, \mu_k)$.  Suppose that
$C(t) = \diag(C_0, \underbrace {0, \ldots, 0}_k)
+ t\cdot\diag(\underbrace{0, \ldots, 0}_{n-k}, C_1) \in T^{reg}$ (i.e. both $C_0$ and $C_1$ are regular and non-degenerate). Then
$$\lim_{t \to 0} B(C(t)) = i_{k}(B(C_0)) \otimes \psi_{n-k}(B(C_1)).$$ \\
2) Let $C_0=\diag(\underbrace{\lambda_1, \ldots, \lambda_1}_{k_1}, \ldots, \underbrace{\lambda_l, \ldots, \lambda_l}_{k_l})$ be a non-degenerate matrix and $C_i=\diag(\underbrace{\mu_{i,1}, \ldots, \mu_{i,k_i}}_{k_i})$ for $i=1,\ldots,l$ such that $\lambda_r\ne\lambda_s$ and $\mu_{i,r}\ne\mu_{i,s}$ for $r\ne s$. Let
\begin{eqnarray*}C(t) := C_0 + t \cdot \diag(C_1, \ldots, C_l).\end{eqnarray*}
Then
$$\lim_{t \to 0} B(C(t)) = B(C_0) \otimes_{\bigotimes\limits_{i=1}^lZ(U(\fgl_{k_i}))} \bigotimes\limits_{i=1}^l \hat F(C_i),$$
where $\hat F(C_i)$ is the shift of argument subalgebra in $U(\fgl_{k_i}) \subset U(\fgl_n) \subset Y(\fgl_n)$ (the copy of $U(\fgl_n)$ in the Yangian $Y(\fgl_n)$ is generated by $t_{ij}^{(1)}$).
\end{thm}
\begin{rem} {\em
Theorems ~\ref{result1} and ~\ref{result2} are equivalent. Indeed, Theorem~\ref{result2} is a particular case of Theorem~\ref{result1}. On the other hand, any degenerate curve $X$ can be obtained from a non-degenerate one by, first, taking a limit from part 1 of Theorem~\ref{result2} ($m$ times where $m$ is the number of irreducible components of $X$ on the way between the marked points $0$ and $\infty$) and, second, taking a limit from part 2 of Theorem~\ref{result2} for each of that $m$ components. So Theorem~\ref{result1} follows from Theorem~\ref{result2}.} 
\end{rem}

\begin{rem} {\em For example, Theorem~\ref{result2} allows to describe explicitly the subalgebra corresponding to the degenerate curve which has exactly $n$ components between the points $0$ and $\infty$ with a unique marked point on each component (i.e. the so-called ``caterpillar curve''). This subalgebra is the same as $\lim\limits_{t_i\to0}B(\diag(1,t_1,t_1t_2,\ldots,t_1\cdot\ldots\cdot t_{n-1}))$. According to Theorem~\ref{result2} this limit subalgebra is the Gelfand-Tsetlin subalgebra of $Y(\fgl_n)$.

The opposite example is the degenerate $2$-component curve such that one component contains $0$ and $\infty$ while the other component contains all other marked points. The subalgebras corresponding to such curve have the form $\lim\limits_{t\to0}B(E+\diag(t\lambda_1,\ldots,t\lambda_n))$ which is according to Theorem~\ref{result2} the subalgebra generated by $B(E)$ and $\hat{F}(\diag(\lambda_1,\ldots,\lambda_n))$.} 
\end{rem}

\section{Proof of the main Theorems.}

Let $C = \diag(\underbrace{\lambda_1, \ldots, \lambda_1}_{k_1}, \underbrace{\lambda_2, \ldots, \lambda_2}_{k_2}, \ldots, \underbrace{\lambda_l, \ldots, \lambda_l}_{k_l})\in T$ be a non-regular element from the maximal torus. Let $d_i(C)$ be the number of homogeneous degree $i$ generators of $F(C)$ (i.e. the multiplicity of the factor $(1-x^i)$ in the Poincare series of $F(C)$).

\begin{prop}\label{bethesize}(Lower bound for the size of Bethe subalgebra)
There is a set of algebraically independent elements of $B(C)$ which consists of $\min(d_i(C)+i-1,n)$ elements of degree $i$ for all $i\in\BZ_{>0}$.
\end{prop}
\begin{proof}
It is enough to check that in $\bar B(C) = \gr B(C) \subset CY(\fgl_n)$ we have enough algebraically independent generators (see $\ref{filtration}$).
The images of $\tau_l(u,C)$  are 
\begin{equation}
\bar\tau_l(u,C) = \sum_{i_1 < \ldots <i_l} \lambda_{i_1} \ldots \lambda_{i_l} \left(\sum_{\sigma \in S_l} \bar{t}_{i_1 i_{\sigma(1)}}(u) \ldots \bar{t}_{i_l i_{\sigma(l)}}(u)\right).
\end{equation}
Now we can change generators $\bar{t}_{ij}^{(p)} \to \tilde t_{ij}^{(p)} = \lambda_i \bar t_{ij}^{(p)}$.
Note that there is  a filtration in $CY(\fgl_n)$ given by
$$\deg \tilde t_{ij}^{(p)} = 1.$$
According to this filtration the leading term of the coefficient of $\bar\tau_l(u,C)$ at $u^{-p}$ for $p \geqslant l$ is $$\sum_{i_1 <\ldots <i_l} \lambda_{i_1} \ldots \lambda_{i_l} \sum_{p_1 + \ldots + p_l = p; \, p_1, \ldots, p_l> 0} \Delta_{i_1 \ldots i_l}^{(p_1 \ldots p_l)}=\sum_{i_1 <\ldots <i_l} \sum_{p_1 + \ldots + p_l = p; \, p_1, \ldots, p_l> 0} \tilde \Delta_{i_1 \ldots i_l}^{(p_1 \ldots p_l)}.$$
Here we set
\begin{equation}
\Delta_{i_1 \ldots i_m}^{(p_1 \ldots p_m)} = \det \begin{pmatrix}
\bar t_{i_1 i_1}^{(p_1)} & \bar t_{i_1 i_2}^{(p_1)} & \dots & \bar t_{i_1 i_m}^{(p_1)} \\

\bar t_{i_2 i_1}^{(p_2)} & \bar t_{i_2 i_2}^{(p_2)} & \dots & \bar t_{i_2 i_m}^{(p_2)} \\
\vdots & \vdots & \vdots & \vdots \\
\bar t_{i_m i_1}^{(p_m)} & \bar t_{i_m i_2}^{(p_m)} & \dots & \bar t_{i_m i_m}^{(p_m)}
\end{pmatrix};\ \tilde \Delta_{i_1 \ldots i_m}^{(p_1 \ldots p_m)} = \det \begin{pmatrix}
\tilde t_{i_1 i_1}^{(p_1)} & \tilde t_{i_1 i_2}^{(p_1)} & \dots & \tilde t_{i_1 i_m}^{(p_1)} \\

\tilde t_{i_2 i_1}^{(p_2)} & \tilde t_{i_2 i_2}^{(p_2)} & \dots & \tilde t_{i_2 i_m}^{(p_2)} \\
\vdots & \vdots & \vdots & \vdots \\
\tilde t_{i_m i_1}^{(p_m)} & \tilde t_{i_m i_2}^{(p_m)} & \dots & \tilde t_{i_m i_m}^{(p_m)}
\end{pmatrix} 
\end{equation}

We see that, for $p\ge l$ the leading terms of such coefficients do not depend on $C$ if we change generators to $\tilde t_{ij}^{(p)}$. Moreover, we see that the leading terms of the coefficients of $\bar \tau_l(u,C)$ at $u^{-l}$ consists of only $\tilde t_{ij}^{(1)}$, coefficients of $u^{-l-1}$ of $\bar\tau_l(u,C)$ contain $\tilde t_{ij}^{(2)}$ and do not contain $ \tilde t_{ij}^{(p)}, p >2$ and so on.  Generally, we have the following

\begin{lem}
The leading term of the coefficient of $\bar \tau_l(u,C)$ at $u^{-l-k}$ does contain $\tilde t_{ij}^{(k+1)}$ and does not contain $\tilde t_{ij}^{(p)}, p > k+1$.
\end{lem}

The leading terms of the coefficients of $u^{-p}, p \leqslant l$ of $\tau_l(u,C)$ are polynomials in $\tilde t_{ij}^{(1)}$ hence can be regarded as elements of $S(\fgl_n)\subset CY(\fgl_n)$ generated by $\tilde t_{ij}^{(1)}$. 

\begin{lem} The leading terms of the coefficients of $u^{-p}, p \leqslant l$ of $\tau_l(u,C)$, $l=1,\ldots,n$ generate the shift of argument subalgebra $F(C)$ in $S(\fgl_n)$. 
\end{lem}

\begin{proof} Set $\bar T^{(1)}:=\sum\limits_{i,j}\bar t_{ij}^{(1)}\otimes e_{ij}\in S(\fgl_n)\otimes Mat_n$. Then the generating function for the leading terms is ${\rm Tr} \, \Lambda^lC(E+T^{(1)}u^{-1})= {\rm Tr} \, \Lambda^l (CE+\tilde T^{(1)}u^{-1})$ which is the generating function for the derivatives of the $l$-th coefficient of the characteristic polynomial. 
\end{proof}

To see that the set of coefficients of $u^{-l-k}$ of $\bar\tau_l(u,C)$ are algebraically independent let us compute differential at point $\tilde t_{ij}^{(1)} = \delta_{i, j-1}$, $\tilde t_{ij}^{(p)} = 0$. It is straightforward computation that
$$\diff \tilde \Delta_{i_1 \ldots i_l}^{(p_1 \ldots p_l)} =
\begin{cases}
(-1)^{l-1} \cdot \diff \tilde t_{i_1, i_1 + l - 1}^{(p-l+1)}, & p_l = p-(l-1), p_{l-1} = 1, \ldots, p_1 = 1, \\
&  i_2 = i_1 + 1, \ldots i_l = i_{l-1} + 1\\
0, & {\text else.}
\end{cases}$$
Hence the differential of the coefficient of $\bar\tau_l(u,C)$ at $u^{-l-k}$ is a non-zero linear combination of $\diff \tilde t_{ij}^{(k+1)}$ with $j-i = l-1$. Hence differentials of all coefficients with $p \geqslant l$ are linearly independent. Therefore if we have algebraic dependence between some coefficients, then it can only consist of coefficients $u^{-p}$ with $p \leqslant l$.

The following tableaux shows the greater number $r$, such that there is some $t_{ij}^{(r)}$ at this coefficient:
\renewcommand{\arraystretch}{1.4}
\begin{center}
 \begin{tabular}{|m{1.5cm} | m{1.2cm} | m{1.2cm} | m{1.2cm} | m{1.2cm} | m{1.2cm}| m{1.2cm}| m{1.2cm}| m{1.2cm}|}
\hline
& $u^{-1}$  &  $u^{-2}$ & $u^{-3}$ & \ldots   &  $u^{-(n-1)}$ & $u^{-n}$ & $u^{-(n+1)}$ &\ldots  \\ 
\hline
$\tau_1(u,C)$ & 1 & 2 & 3 & \ldots & $n-1$ & $n$ & $n+1$ & \ldots \\ 
\hline
$\tau_2(u,C)$ & 1 & 1  & 2 & \ldots & $n-2$ & $n-1$ & $n$ & \ldots \\
\hline
$\tau_3(u,C)$ & 1 & 1 & 1 & \ldots & $n-3$ & $n-2$ & $n-1$ & \ldots \\
\hline
$\ldots$ & & & & & & & & \\
\hline
$\tau_n(u,C)$ & 1 & 1 & 1 & \ldots &  1 & 1 & 2 & \ldots\\ 
\hline
\end{tabular}
\end{center}

On the other hand, the  coefficients of $u^{-p}, p \leqslant l$ of $\tau_l(u,C)$, $l=1,\ldots,n$ generate the shift of argument subalgebra $F(C)$ in $S(\fgl_n)$ hence we have the desired number of algebraically independent generators. 

%Note that coefficients of different degrees are algebraically independent:
%according to filtration $\deg t_{ij}^{(p)} = p$ they have different degrees.
%So we need to check how generators of coefficients of degrees $p < l$ depends on $C$.

%Hence we obtain the result.
\end{proof}

Recall the following fact from \cite{molev2}:

\begin{lem}
\label{togd}
\cite{molev2}[Theorem 1.10.7]
${\rm qdet} \, T(u) \cdot \omega_n(t^{j_{m+1}\ldots j_n}_{i_{m+1}\ldots i_n}(-u+n-1)) = \sgn p \cdot \sgn q \cdot t^{i_1 \ldots i_m}_{j_1\ldots j_m}(u)$, where $p = \left(\begin{smallmatrix}
1 & 2 & \ldots & n \\
i_1 & i_2 & \ldots & i_n \end{smallmatrix}\right), q = \left(\begin{smallmatrix}
1 & 2 & \ldots & n \\
j_1 & j_2 & \ldots & j_n \end{smallmatrix}\right) \in S_n$.
\end{lem}

This implies the following

\begin{lem}
\label{betheImage}
Suppose that $C \in T$ is non-degenerate. Then
$\omega_n(B(C)) = B(C^{-1})$.
\end{lem}
\begin{proof}
Note that the series ${\rm qdet} \, T(u)$ is invertible so multiplying by ${\rm qdet} \, T(u)$ is a bijection on a set of generators of $B(C)$. Moreover $T(u) \to T(u-c)$ is invertible automorphism of $Y(\fgl_n)$ so we can consider $\tau_k(-u+n-1)$ instead of $\tau_k(u,C)$ as a generators of $B(C)$.
By Lemma \ref{togd}
\begin{multline*}
(\det C)^{-1} \cdot \qdet T(u) \cdot \omega_n(\tau_k(-u+n-1,C)) = \\ =(\det C)^{-1} \cdot \sum_{1 \leqslant a_1< \ldots < a_k \leqslant n} \lambda_{a_1} \ldots \lambda_{a_k} \qdet T(u) \cdot \omega_n(t_{a_1, \ldots, a_k}^{a_1,\ldots,a_k}(-u+n-1)) = \\ = (\det C)^{-1} \sum_{1 \leqslant a_1< \ldots < a_k \leqslant n} \lambda_{a_1} \ldots \lambda_{a_k} t_{a_{k+1}, \ldots, a_n}^{a_{k+1},\ldots,a_n}(u) =\\=  \sum_{1 \leqslant a_1< \ldots < a_k \leqslant n} \lambda_{a_{k+1}}^{-1} \ldots \lambda_{a_n}^{-1} t_{a_{k+1}, \ldots, a_n}^{a_{k+1},\ldots,a_n}(u) = \tau_{n-k}(u, C^{-1}).
\end{multline*}
\end{proof}

\subsection{Proof of Theorem ~\ref{result2}, part 1.}
\label{pf1} We have $i_{k}(B(C_0))$ in the limit subalgebra. To its generators one can just take the standard generators of the Bethe subalgebra
$\tau_k(u,C_0 + t \cdot C_1), 1 \leqslant k \leqslant n,$ and put $t = 0$. 

On the other hand, from Lemma \ref{betheImage} we have the following:
\begin{multline*}
\lim_{t \to 0} B(a(t)) = \lim_{t \to 0} \omega_n(B(a(t)^{-1})) =  \omega_n(\lim_{t \to 0} B(a(t)^{-1})) \supset \\ \supset (\omega_n \circ \varphi_{n-k}) (B(C_1^{-1}))  = (\omega_n \circ \varphi_{n-k} \circ \omega_{k}) (B(C_1)) = \psi_{n-k}(B(C_1)).
\end{multline*}
So the limit subalgebra contains $i_k(B(C_0)) \cdot \psi_{n-k}(B(C_1))$. The latter is $i_k(B(C_0)) \otimes \psi_{n-k}(B(C_1))$ due to Lemma \ref{commuteyang}.
Since $C_0$ and $C_1$ are regular elements of $gl_k$ and $gl_{n-k}$ respectively, the tensor product has $i_k(B(C_0)) \otimes \psi_{n-k}(B(C_1))$ has the same Poincare series as generic Bethe algebra, therefore the limit is in fact equal to $i_k(B(C_0)) \otimes \psi_{n-k}(B(C_1))$. So Theorem~\ref{result2} (1) is proved.

\begin{prop}
\label{incl}
Let $C\in T$ be any non-degenerate diagonal matrix. Set $C^{(k)}:=\diag(C,\underbrace{0,\ldots,0}_k)$. Then $\eta_k(B(C) \otimes A_0) \subset \hat F(C^{(k)})$ for any non-degenerate $C \in \fh$ and $k \in \mathbb{N}$.
Moreover, the restrictions of $\{\eta_k\}$ to $B(C)$ give an asymptotic isomorphism of $B(C)$ and $\hat{F}(C^{(k)})$, $k\to\infty$.
The same is true for all limit algebras of the family $B(C)$, i.e. we have $\eta_k(\lim\limits_{t\to0}B(C(t)) \otimes A_0) \subset \lim\limits_{t\to0}\hat F(C(t)^{(k)})$.
\end{prop}
\begin{proof}
We have $\eta_k = \pi_{n+k} \circ \omega_{n+k} \circ i_k$. Consider $a(t) = \diag(C,\underbrace{0, \ldots, 0}_k)+ t\cdot\diag(0, \ldots , 0, C^{\prime})$ and subalgebra $B(a(t))$ in $Y(\fgl_{n+k})$. Then by Theorem~\ref{result2} (1) the limit subalgebra $\lim\limits_{t\to0}B(a(t))$ is equal to $i_n(B(C^{(k)})) \otimes \psi_{n-k}(B(C^{\prime}))$. Note that $B(C^{(k)}) = i_k(B(C))$. We have  
\begin{multline*}
\pi_{n+k} \circ \omega_{n+k}(\lim_{t \to 0} B(a(t)) \otimes A_0) \subset \lim_{t \to 0}(\pi_{n+k} \circ \omega_{n+k}(B(a(t)) \otimes A_0)  = \\ = \lim_{t \to 0} \hat F(a(t))) =  \hat F(C^{(k)}) \otimes_{Z(U(\fgl_k))} \hat F(C^{\prime}).
\end{multline*}
We use here the fact that $\pi_{n+k}(B(a(t)) = \hat F(a(t)^{-1})$, see ~\cite{taras2}.
But $\eta_k(B(C)) \subset U(\fgl_{n+k})^{\fgl_k}$ therefore $\eta_k(B(C) \otimes A_0) \subset \hat F(C^{(k)}) \otimes_{Z(U(\fgl_k))} \hat F(C^{\prime})\cap U(\fgl_{n+k})^{\fgl_k} = \hat F(C^{(k)})$.
The fact that $\{\eta_k\}$ is asymptotic isomorphism follows from Lemma~\ref{bethesize}.
For limit algebras it follows from the fact that $\eta_k(\lim_{t \to 0}(B(C(t))\otimes A_0)) \subset \lim_{t \to 0} F(C(t)^{(k)},0, \ldots,0)$ and the fact that limit algebras have the same Poincare series.
\end{proof}

\subsection{Proof of Theorem ~\ref{result2}, part 2.}
Analogously to \ref{pf1}, we have $B(C_0)$ in the limit subalgebra. Also we know that
$$\eta_k(\lim_{t \to 0} (B(C(t))\otimes A_0) \subset \lim_{t \to 0} \hat{F}(C(t)^{(k)}) = \hat F(C_0^{(k)}) \otimes_{ZU(\mathfrak{z}(C_0))
%Z(U(\fgl_{n+k}))^{\fgl_k}
} \hat F(C_1).$$
We see that $B(C_0) \otimes_{Z(U(\mathfrak{z}({C_0})))} \hat F(C_1) \otimes A_0$ maps to $\hat F(C_0^{(k)}) \otimes_{ZU(\mathfrak{z}(C_0))} \hat F(C_1)$.
Moreover,$\eta_k$ is asymptotic isomorphism and  restriction of $\eta_k$ to   $U(\mathfrak{z}(C_0)) \subset Y(\fgl_{n})$ is identity, so we have $\hat F(C_1)$ in the limit subalgebra as well. But using Proposition~\ref{bethesize} we see  that the Poincare series of  $B(C_0) \otimes_{Z(Y(\fgl_k))} \otimes \hat F(C_1)$ coincides with that of $B(C(t))$ for generic $t$.
%Z(U(\fgl_{n+k}))^{\fgl_k}
\subsection{Proof of Theorem ~\ref{result}}
From Proposition ~\ref{incl} we know that family \{$\eta_k$\} is asymptotic isomorphism between $B(C) \otimes A_0$ and $\hat F(\diag(C,\underbrace{0,\ldots,0}_k))$. According to Lemma ~\ref{l36} all algebras $\hat F(\diag(C,\underbrace{0,\ldots,0}_k))$ form the family $\mathfrak{A}$. Limit subalgebras of this family are parameterized by $\ol{M_{0,n+2}}$ for any $k > 0$. Therefore limit Bethe subalgebras are parameterized by $\ol{M_{0,n+2}}$ as well.
Additionally, all limit algebras of the family $\mathfrak{A}$ are free and maximal commutative by Lemma ~\ref{l36}. Therefore all limit Bethe subalgebras in the Yangian are free and maximal commutative as well. 

\subsection{Proof of Theorem ~\ref{result1}} We have already proved the first part of the Theorem. 
We only need to see that $i_{k_0}(B(C))$ contains the center of every $U(\fgl_{k_\lambda})\subset i_{k_0}(Y(\fgl_{n-k_0}))$.
We know that the image $\eta_k(B(C) \otimes A_0)$ contains the center of every $U(\fgl_{k_\lambda})$ for any $k$.
Also we know that $A_0$ maps to the center of $U(gl_{n+k})$ by $\eta_k$. But from ~\cite[Lemma 10.1]{Kn} it follows that the generators of the center of $U(gl_{n+k})$ and of the centers of all $U(\fgl_{k_\lambda})$ are algebraically independent.
Using last time the fact that $\eta_k$ is asymptotic isomorphism we obtain the result.

\bigskip
\footnotesize{
{\bf Leonid Rybnikov} \\
National Research University
Higher School of Economics,\\ Russian Federation,\\
Department of Mathematics, 6 Usacheva st, Moscow 119048;\\
Institute for Information Transmission Problems of RAS;\\
{\tt leo.rybnikov@gmail.com}} \\
\\
\footnotesize{
{\bf Aleksei Ilin} \\ National Research University
Higher School of Economics, \\
Russian Federation,\\
Department of Mathematics, 6 Usacheva st, Moscow 119048;\\
{\tt alex\_omsk@211.ru}}

\end{document}